\documentclass[12pt]{amsart}
\usepackage{amsmath}
\usepackage{amstext}
\usepackage{amsfonts}
\usepackage{amssymb}
\usepackage{amsthm}
\usepackage{cite}
\usepackage{amsaddr}
\usepackage{color}
\usepackage{enumitem}
\usepackage{hyperref}
\usepackage{mathtools}
\usepackage{microtype}
\usepackage{thmtools}
\usepackage{tikz-cd}
\allowdisplaybreaks[3]

\mathtoolsset{centercolon}

\definecolor{darkblue}{rgb}{0.0,0.0,0.3}
\hypersetup{colorlinks,breaklinks,allcolors=darkblue}

\theoremstyle{plain}
\newtheorem{theorem}{Theorem}
\newtheorem{proposition}{Proposition}
\newtheorem{lemma}{Lemma}

\theoremstyle{definition}

\newtheorem{example}{Example}
\newtheorem{observation}{Observation}

\begin{document}

\title[The Extension of Semigroups on Operator Systems]{The Extension of Unital Completely Positive Semigroups on Operator Systems to Semigroups on $C^*$-algebras}

\author[V.I. Yashin]{Vsevolod I. Yashin}
\address{Steklov Mathematical Institute of Russian Academy of Sciences, Gubkina 8 St., Moscow 119991, Russia}
\email{viyashin@protonmail.com\vspace{-2ex}}

\begin{abstract}
The study of open quantum systems relies on the notion of unital completely  positive semigroups on $C^*$-algebras representing physical systems. The natural generalisation would be to consider
the unital completely positive semigroups on operator systems. We show that any continuous unital completely positive semigroup on matricial system can be extended to a semigroup on a finite-dimensional $C^*$-algebra, which is an injective envelope of the matricial system. In case the semigroup is invertible, this extension is unique.
\end{abstract}

\thanks{Supported by Russian Science Foundation under the grant
no. 19-11-00086.}
\maketitle
\tableofcontents


\section{Introduction} \label{sec:introduction}

The theory of open quantum dynamics is an important and well-developed  field of study \cite{Breuer_Petruccione_2002}. The Markovian dynamics of physical systems may be seen as semigroups of
operators on Banach spaces. Such semigroups were studied for a long time due to the abundance of applications \cite{Engel_Nagel_2000}. The positive operator semigroups have an especially good behavior and a physical meaning \cite{Batkai_Fijvacz_Rhandi_2017}.

The development of quantum mechanics has shown that it is natural to consider  quantum dynamics to be completely positive \cite{Kraus_1983}. At the same time, the notion of complete positivity was widely studied by functional analysts, resulting in the theory of operator systems \cite{Choi_Effros_1977, Paulsen_2003}. The operator systems may be naturally called the ``non-commutative'' versions of Banach spaces with order \cite{Helemskii_2010}. Adopting this idea, we may regard the completely positive unital semigroups on operator systems to be the natural framework for studying quantum Markovian dynamics.

Earlier, the completely positive operator semigroups were studied in case of the $C^*$-algebras, leading to a number of results, probably the most famous being GKSL characterisation of generators
\cite{Lindblad_1976,Gorini_Kossakowski_Sudarchan_1976} and Bhat's dilation theorem \cite{Bhat_1996, Bhat_1999}. In this work we show that the theory of positive semigroups on operator systems can be
included into the theory of semigroups on $C^*$-algebras, at least in finite-dimensional case. That is, we show that any evolution on matricial system can be extended to an evolution on finite-dimensional $C^*$-algebra.

Here is a structure of a paper. In Section \ref{sec:preliminaries} we fix  the notation used and remind basic facts about operator systems and one-parameter operator semigroups. In Section \ref{sec:main_result} we present the results which state that any continuous unital completely positive semigroup on a matricial system can be extended to a semigroup on some $C^*$-algebra, and that the continuous unital completely positive one-parameter groups are extended uniquely. The discussion on some implications of the result is given in Section \ref{sec:discussion}.

\section{Preliminaries} \label{sec:preliminaries}

In this section we fix the notation and remind some notions necessary  for the rest of the paper. The theory of ordered vector spaces may be found in \cite{Aliprantis_Chalambos_Tourky_2007}, the theory of operator systems and operator spaces is given in books \cite{Paulsen_2003, Effros_Ruan_2000, Pisier_2003}, and also a paper \cite{Paulsen_Todorov_Tomforde_2010}, where the connection between ordered vector spaces and operator systems is established. For the introduction to the theory of operator spaces, which are the ``non-commutative'' versions of Banach spaces, the reader may consult books \cite{Pisier_2003, Effros_Ruan_2000, Paulsen_2003, Helemskii_2010}. A deep and clear exposition on the theory of operator semigroups is contained in a book \cite{Engel_Nagel_2000}. The theory of positive operator semigroups is presented in \cite{Batkai_Fijvacz_Rhandi_2017}.

\subsection{Ordered vector spaces} \label{subsec:ordered_vector_spaces}

Let $V$ be a complex vector space. We will be interested in the case $V$ is Banach, i.e. is endowed with topology and norm $\lVert\cdot\rVert$ and is complete. If $V$ is equipped with a complex-conjugate involution $v \mapsto v^*$, then it is called \emph{$*$-vector space}. The space of Hermitian elements is denoted as $V_h = \left\{ v\in V \, \middle| \, v = v^* \right\}$. A \emph{cone} $C$ is a subset of $V$ closed under addition and multiplication by positive numbers. We will be especially interested in the case when a cone $C$ is
\begin{enumerate}
  \item Hermitian: $C \subseteq V_h$,
  \item non-degenerate: $C \cap - C = \{0\}$,
  \item generating: the linear span equals the whole space, $\mathrm{span}(C) = V$,
  \item closed in the topology of $V$.
\end{enumerate}
We will call a vector space $V$ with such a cone an \emph{ordered $*$-vector space}, and we will denote the cone $C$ as $V^+$. The elements of a cone are called \emph{positive}. For two vectors $u$ and $v$ in ordered vector space $V$ it holds $u\leq v$ exactly when $v - u \in V^+$. It is easy to check that this relation is a partial ordering, which is also translation invariant and homogeneous when multiplying by positive numbers.

A \emph{vector lattice} is an ordered vector space in which for any positive vectors $u, v \in V^+$ there is a minimal upper bound $u \vee v$. The operation of taking an absolute value is defined as $|v| = v\vee -v$. The vector lattice is called \emph{Banach lattice} if $|u| \leq |v|$ implies $\lVert u\rVert \leq \lVert v\rVert$. The usual examples of Banach lattices are function spaces with pointwise ordering, such as $L^p(\Omega,\mu)$ and $\mathcal{C}(\Omega)$.

\begin{example}
It is instructive to consider the space of continuously differentiable functions $\mathcal{C}^1([0,1])$ with pointwise ordering. This is an ordered Banach space, but is not a Banach lattice, because the pointwise maximum of two positive functions $f\vee g$ is usually not differentiable. However, the order gives rise to a norm on this space, which is the supremum norm $\lVert \cdot\rVert_\infty$. The completion by this norm constitutes the Banach lattice $\mathcal{C}([0,1])$.
\end{example}

An operator $\phi : V \rightarrow W$ between  ordered spaces is called \emph{positive} when it preserves positive elements, that is $\phi(V^+) \subseteq W^+$. If $\phi : V \rightarrow W$ is a positive map and has a positive inverse map, then it is called an \emph{order isomorphism} of $V$ and $W$. If a positive map $i$ is an isomorphism of $V$ to a subspace of $W$, i.e. $i[V^+] = i[V]\cap W^+$, we will call $i$ an \emph{embedding} and denote $i : V \hookrightarrow W$. Note that the image of positive maps is not necessarily closed, so the injective positive maps are not always embeddings.
\begin{example}
  Let us consider a map $\phi : \mathcal{C}([0,1]) \rightarrow \mathcal{C}([0,1])$, given by
  \begin{equation*}
    \phi[f](T) =
    \begin{cases}
      f(0), & \text{ if } T = 0, \\
      \frac{1}{T}\int_0^T f(t) dt, & \text{ if } 0 < T \leq 1.
    \end{cases}
  \end{equation*}
  This map is positive and preserves unit function, but it's image consists of differentiable maps, so it is not closed in $\mathcal{C}([0,1])$.
\end{example}
The positive map $p$ from $W$ to $V$ is called a \emph{quotient map} if it is surjective and $p[W^+] = V^+$, such maps will be denoted as $p : W \twoheadrightarrow V$.

\subsection{Function systems} \label{subsec:function_systems}

Let $V$ be an ordered vector space. We say that the element $e \in V^+$ is an \emph{order unit} if for every $v \in V_h$ there exists $r>0$ such that $v \leq r e$. The order unit is called \emph{Archimedean} if any vector $v \in V_h$ is positive if and only if $v + re$ is positive for all $r>0$. The spaces with Archimedean order unit are called \emph{function systems}, or \emph{Archimedean order unit (AOU)} spaces \cite{Paulsen_Tomforde_2009}.

We will denote $\langle v, \rho\rangle = \rho(v)$ for  vector $v\in V$ and functional $\rho\in V^*$ a canonical pairing of vector spaces $V$ and $V^*$. A \emph{state} $\rho\in V^*$ is a unital positive functional on a function system, which means $\langle e,\rho\rangle = 1$ and $\langle v,\rho\rangle \geq 0$ for all $v\in V^+$. The space of states on function systems $V$ is denoted $S(V)$. The set $S(V)$ is convex and compact in weak-$*$ topology, the extremal points of it are called \emph{pure states}.

The Archimedean order unit naturally defines the topology on the function system. The norm $\lVert\cdot\rVert_h$ on real vector space $V_h$ is defined as
\begin{equation*}
  \lVert v\rVert_h = \inf\left\{ r > 0 \, \middle| \, -r e \leq v \leq r e \right\}.
\end{equation*}
An \emph{order norm} $\lVert\cdot\rVert$ is a norm on $V$ which satisfies $\lVert v\rVert = \lVert v^*\rVert$ and $\lVert v\rVert = \lVert v\rVert_h$ if $v = v^*$. The minimal order norm is
\begin{equation*}
  \lVert v\rVert = \sup\left\{ |\langle v,\rho\rangle| \, \middle| \, \rho \in S(V) \right\}.
\end{equation*}
It fact, all the order norms are equivalent to the minimal one \cite{Paulsen_Tomforde_2009} and all of them induce the same topology, which is generated by absolutely convex subsets that absorb order intervals (see \cite[Chapter 2]{Aliprantis_Chalambos_Tourky_2007}).

The map $\phi : V\rightarrow W$ is called \emph{unital} if it preserves the order unit, i.e. $\phi[e_V] = e_W$. If $\phi : V \rightarrow W$ is a unital positive map and has a positive inverse map, then it is called an \emph{unital order isomorphism} of $V$ and $W$. If $i$ is a unital isomorphism of $V$ to a subspace of $W$, we will call $i$ an \emph{embedding} and denote $i : V \hookrightarrow W$. The quotient map $p$ between $W$ and $V$ is a unital surjective map such that $p[W^+] = V^+$. We will denote quotient maps as $p : W \twoheadrightarrow V$.

R.V.~Kadison in the work \cite{Kadison_1951} has proven that any function system may be represented as a closed subspace in an algebra of continuous functions $\mathcal{C}(\Omega)$ on a Hausdorff compact $\Omega$ via unital order embedding. To be exact, this map is given as
\begin{equation*}
  V \rightarrow \mathcal{C}(S(V)), \qquad v \mapsto \left( f \mapsto f(v) \right).
\end{equation*}
It is additionally known that if the function system $V$ is a lattice, then the pure states are the lattice homomorphisms, and $V$ is equivalent to $\mathcal{C}(\Omega)$ for the space $\Omega$ of pure states (see \cite{Kakutani_1941} and \cite[Volume 2]{Lindenstrauss_Tzafriri_1996}).

\subsection{Operator spaces}

Let $\mathbb{M}_{n,m} = \mathbb{M}_{n,m}(\mathbb{C})$ be a space of  $n\times m$ matrices with complex coefficients. For a $*$-vector space $V$ let $\mathbb{M}_{n,m}(V)$ be a space of $n \times m$ matrices with elements from $V$. Equivalently, we may think of $\mathbb{M}_{n,m}(V)$ as a tensor product $\mathbb{M}_{n,m} \otimes V$. The space $\mathbb{M}_{n}(V) = \mathbb{M}_{n,n}(V)$ is called an \emph{$n$-matrix level} of $V$. An involution
\begin{equation*}
  (\cdot)^* : (v_{i,j})_{i,j} \mapsto (v_{j,i}^*)_{i,j}
\end{equation*}
makes $\mathbb{M}_{n}(V)$ into a $*$-vector space. We define the multiplication of $v \in \mathbb{M}_{n,m}(V)$ by matrices $\alpha \in \mathbb{M}_{s,n}$ on the left and by $\beta \in \mathbb{M}_{m,p}$ on the right as
\begin{equation*}
  (\alpha v \beta)_{k,l} = \sum_{i,j} \alpha_{k,i} v_{i,j} \beta_{j,l} \in \mathbb{M}_{s,p}(V) .
\end{equation*}
A \emph{matrix norm} $\{ \lVert\cdot\rVert_n \}_n$ is a sequence of norms $\lVert\cdot\rVert_n : \mathbb{M}_{n} \rightarrow \mathbb{R}_+$ on each matrix level, such that the axioms of Ruan \cite{Ruan_1988} hold:
\begin{itemize}
  \item[(R1)] $\lVert\alpha v \beta\rVert_m \leq \lVert\alpha\rVert \lVert v\rVert_n \lVert\beta\rVert$ for all $v \in \mathbb{M}_{n}(V), \; \alpha \in \mathbb{M}_{m,n}, \; \beta \in \mathbb{M}_{n,m}$,
  \item[(R2)] $\lVert x \oplus y\rVert_{n+m} = \max{ \{ \lVert x\rVert_n, \lVert y\rVert_m \} }$ for all $x \in \mathbb{M}_{n}(V), \; y \in \mathbb{M}_{m}(V)$.
\end{itemize}
If $V$ is a Banach space with a matrix norm, then it is called an \emph{operator space}.

Suppose $\phi : V \rightarrow W$ is a bounded linear map between operator spaces.  Denote the $n$th amplification $\phi^{(n)} : \mathbb{M}_{n}(V) \rightarrow \mathbb{M}_{n}(W)$ a map $\mathrm{id}_n \otimes \phi$. Thinking in terms of matrices, it acts as
\begin{equation*}
  \phi^{(n)}
  \left(
    \begin{bmatrix}
      v_{11} & \cdots & v_{1n} \\
      \vdots & \ddots & \vdots \\
      v_{n1} & \cdots & v_{nn}
    \end{bmatrix}
  \right)
  =
  \begin{bmatrix}
    \phi[v_{11}] & \cdots & \phi[v_{1n}] \\
    \vdots & \ddots & \vdots \\
    \phi[v_{n1}] & \cdots & \phi[v_{nn}]
  \end{bmatrix}.
\end{equation*}
For the operator norms at each matrix level it holds
\begin{equation*}
  \lVert\phi\rVert \leq \lVert\phi^{(2)}\rVert \leq \cdots \leq \lVert\phi^{(k)}\rVert \leq \cdots
\end{equation*}
The limit of this norms defines a \emph{completely bounded norm}
\begin{equation*}
  \lVert\phi\rVert_{cb} = \sup_n \lVert\phi^{(n)}\rVert
\end{equation*}
and if $\lVert\phi\rVert_{cb} \leq \infty$, then $\phi$ is called \emph{completely bounded} map. It is known that the bounded maps between ``finite-dimensional'' operator spaces are completely bounded
\begin{lemma}[Smith, \cite{Smith_1983}]
  Let $V$ be a matricially normed space and $\phi : V \rightarrow \mathbb{M}_n$
  a bounded map. Then $\lVert\phi\rVert_{cb} = \lVert\phi^{(n)}\rVert$.
\end{lemma}

The space of completely bounded maps will be denoted $\mathcal{CB}(V,W)$. The map $\phi$ is called \emph{completely contractive} if $\lVert\phi\rVert_{cb}\leq 1$. The map $\phi$ is called \emph{complete isometry} if all amplifications are isometric, i.e. $\lVert\phi^{(n)}[v]\rVert_{n}= \lVert v\rVert_{n}$ for all $v\in\mathbb{M}_{n}(V)$. If it is also surjective, then $\phi$ is called \emph{complete metric isomorphism}.

\subsection{Operator systems} \label{subsec:operator_systems}

The containing unit $*$-subspaces of general $C^*$-algebras are called \emph{(concrete) operator systems} \cite{Paulsen_2003}. Similar to Kadison's result, M.-D.~Choi and E.G.~Effros were able to characterise operator systems by their order structure \cite{Choi_Effros_1977}.

A family of cones $\{C_n\}_n$ such that $C_n \subseteq \mathbb{M}_{n}(V)_h$ constitutes a \emph{matrix cone} on $V$ if
\begin{equation*}
  \text{for each} \;\alpha \in \mathbb{M}_{n,m} \; \text{it is true that} \; \alpha^* C_n \alpha \subseteq C_m .
\end{equation*}
We will also write $\mathbb{M}_{n}(V)^+$ for $C_n$. The matrix cones provide $V$ with \emph{matrix order}.

An order unit $e$ is called \emph{matrix order unit} if the element $e^{(n)} \in \mathbb{M}_{n}(V)$:
\begin{equation*}
  e^{(n)} =
  \begin{bmatrix}
    e &        & \\
      & \ddots & \\
      &        & e
  \end{bmatrix}
  = I_n \otimes e
\end{equation*}
is an order unit $\mathbb{M}_{n}(V)$ on each matrix level. If the matrix order unit $e^{(n)}$ is also Archimedean for each $\mathbb{M}_{n}(V)$, then it is called \emph{Archimedean matrix order unit}. An \emph{(abstract) operator system} is a matrix ordered $*$-vector space with Archimedean matrix order unit. Let us note that the matrix cone endows $V$ with a matrix norm structure, given by
\begin{equation*}
  \lVert v\rVert_n = \inf\left\{ r > 0 \; \middle| \, \begin{bmatrix} r I_n & v \\ v^* & r I_n \end{bmatrix} \in \mathbb{M}_{2n}(V)^+ \right\}.
\end{equation*}

Let $V$ and $W$  be matrix ordered vector spaces and  $\phi : V \rightarrow W$ a linear map between them. A map $\phi$ is called \emph{completely positive} if each $\phi^{(n)}$ is positive, i.e. it sends $\mathbb{M}_{n}(V)^+$ into $\mathbb{M}_{n}(W)^+$. We will denote the space of unital completely positive maps between $V$ and $W$ as $\mathrm{UCP}(V,W)$. The unital completely positive maps are always completely contractive and $\lVert\phi\rVert_{cb} = 1$. On the other hand, all completely contractive unital maps are completely positive.

If a unital completely positive map $\phi : V \rightarrow W$ has a completely positive inverse map, then it is called a \emph{unital complete order isomorphism}. The unital completely positive map $i : V \hookrightarrow W$ is called a \emph{unital complete order embedding} if it is an isomorphism of $V$ to a subsystem of $W$. The map $p : V\twoheadrightarrow W$ is a \emph{unital completely quotient map} if it is quotient on each matrix level. The theory of quotients for operator systems is presented in \cite{Kavruk_Paulsen_Todorov_Tomforde_2013}.

Let $\mathcal{H}$ be a Hilbert space and $\mathcal{B}(\mathcal{H})$ the algebra of bounded operators on it. The representation theorem of Choi and Effros states that each abstract operator system is completely order isomorphic to a $*$-subspace of $\mathcal{B}(\mathcal{H})$ that contains identity operator $I$.

\begin{theorem}[Choi--Effros, \cite{Choi_Effros_1977}]
  Let $V$ be an abstract operator system. Then there exists a Hilbert space $\mathcal{H}$, a concrete operator system $S \subseteq \mathcal{B}(\mathcal{H})$ and a unital complete order isomorphism $\phi : V \rightarrow S$.
\end{theorem}

This isomorphism is obtained by an embedding
\begin{equation*}
  V \rightarrow \bigoplus_{n \in \mathbb{N}} \bigoplus_{\varphi \in \mathrm{UCP}(V,\mathbb{M}_{n})} \mathbb{M}_{n}, \quad
  v \mapsto (\varphi(v))_\varphi .
\end{equation*}

The operator system $V$ that is embeddable into a finite-dimensional $\mathbb{M}_{n}$ is called a \emph{matricial system} \cite{Arveson_2010}. For matricial systems, in order to check the complete positivity of a map $\phi$ from or into $V$ it is sufficient to check the positivity of $\phi^{(n)}$. This statement is often refereed to as \emph{Smith's lemma} \cite{Smith_1983}.

We pay reader's attention to the distinction: the function systems should be treated as the spaces of (affine) functions, while the operator systems should be treated as the spaces of operators. The operator systems are endowed with an additional structure of matrix order, which acts similarly to a ``topology''. Each operator system is a function system, but on the given function system there may be a lot of different matrix order structures \cite{Paulsen_Todorov_Tomforde_2010}.

\subsection{Operator semigroups} \label{subsec:operator_semigroups}

Let $V$ be a Banach space, $\mathcal{B}(V)$ is an algebra  of bounded operators on $V$ and let $\mathbb{R}_+$ be a semigroup of positive real numbers. A \emph{one-parameter operator semigroup} is a semigroup homomorphism $ \Phi : \mathbb{R}_+ \rightarrow \mathcal{B}(V) $. Equivalently, an operator semigroup is a set of operators $\{\Phi(t)\}_{t\geq 0}$, such that
\begin{equation*}
  \Phi(0) = \mathrm{id}, \quad \Phi(s) \Phi(t) = \Phi(t+s),
\end{equation*}
where $\mathrm{id}$ is an identity operator and $t,s\in\mathbb{R}_+$. An operator semigroup  is called \emph{strongly continuous} if for each $v\in V$ the orbit of this element is continuous in strong operator topology, which means
\begin{equation*}
  \lim_{t \downarrow 0} \left\lVert \Phi(t) v - v \right\rVert = 0 \; \text{ for all } v \in V.
\end{equation*}
An operator semigroup is called \emph{uniformly continuous} when the orbit is continuous in norm operator topology, that is
\begin{equation*}
  \lim_{t \downarrow 0} \left\lVert \Phi(t) - I \right\rVert_{\mathcal{B}(V)} = 0,
\end{equation*}

For any strongly continuous semigroup there is a densely defined closable operator $A : \mathcal{D}(A) \rightarrow V $, called a \emph{generator} of a semigroup. It is given by a formula
\begin{equation*}
  Av = \lim_{t \downarrow 0} \frac{\Phi(t) v - v}{t} \; \text{ for all } v\in\mathcal{D}(A).
\end{equation*}
It is useful to examine the resolvent of an operator $R(\lambda,A)$, defined on a resolvent set
\begin{equation*}
  \rho(A) = \mathbb{C} \setminus \sigma(A) = \left\{ \lambda\in\mathbb{C} \, \middle| \, \lambda - A \text{ has bounded inverse }\right\}
\end{equation*}
and given by a formula
\begin{equation*}
  R(\lambda,A) = (\lambda - A)^{-1}.
\end{equation*}
The resolvent function satisfies the \emph{Hilbert's identity}
\begin{equation} \label{eq:Hilberts_identity}
  \frac{R(\lambda,A)-R(\omega,A)}{\lambda-\omega} = - R(\lambda,A)R(\omega,A)
\end{equation}
for all $\lambda,\omega \in \rho(A)$. This identity implies that the resolvent $R(\lambda,A)$ is holomorphic on the parameter $\lambda$. Also, for strongly continuous semigroups it holds that
\begin{equation*}
  \lim_{\lambda\rightarrow\infty}\lVert \lambda R(\lambda,A) v - v \rVert = 0 \text{ for all } v\in V.
\end{equation*}

The semigroup $\Phi$ is uniformly continuous exactly in case $A$ is a bounded operator. In case $A$ is unbounded, the semigroup is strongly continuous.

Studying spectral properties of a generator $A$, one may find a lot of information about the semigroup $\Phi$. To a generator $A$ we associate a \emph{spectral bound} $s(A)$, which is the largest real part of a spectrum
\begin{equation*}
  s(A) = \mathrm{sup}\left\{ \mathrm{Re}(\lambda) \, \middle| \, \lambda \in \sigma(A) \right\}.
\end{equation*}
The \emph{growth bound} of a semigroup is a number
\begin{equation*}
  \omega_0 = \inf_{t > 0} \frac{\log\lVert\Phi(t)\rVert}{t}.
\end{equation*}
The spectral bound and growth bound correspond as
\begin{equation*}
  -\infty \leq s(A) \leq \omega_0 < +\infty.
\end{equation*}
For $\lambda$ such that $\mathrm{Re}(\lambda) > \omega_0$ the resolvent is a Laplace  integral (in Bochner's sense) of $\Phi$:
\begin{equation} \label{eq:resolvent_as_Laplace}
  R(\lambda,A) = \int_0^\infty e^{- \lambda t} \Phi(t) d t .
\end{equation}
The operator semigroup is called \emph{contractive semigroup} if every operator  in a semigroup is contractive, that is $\lVert\Phi(t)\rVert \leq 1$ for all $t \geq 0$. For contractive semigroups it is evident that $\omega_0 \leq 0$. The generators of contractive semigroups are characterised by theorems of Lummer-Phillips and Hille-Yosida \cite{Engel_Nagel_2000}.

An operator semigroup $\Phi$ on an ordered vector space $V$ is called \emph{positive} if it consists of positive operators,
\begin{equation*}
  \Phi(t)v \geq 0 \; \text{ for all }\; t \geq 0, v \in V^+.
\end{equation*}
or equivalently
\begin{equation*}
  \langle\Phi(t)v,\rho\rangle \geq 0 \; \text{ for all } \; t\geq 0, v\in V^+, \rho\in (V^*)^+.
\end{equation*}

If the semigroup is positive and unital, it holds $\omega_0 = s(A) = 0$ and the resolvent $R(\lambda,A)$ is defined for all $\lambda \in \mathbb{C}$ such that $\mathrm{Re}(\lambda) > 0$.

The generators of positive semigroups are characterised using the notion of conditional (off-diagonal) positivity. The operator $A : \mathcal{D}(A) \rightarrow V$ on a function system $V$ is \emph{conditionally positive} if and only if
\begin{equation*}
  \langle v,\rho\rangle = 0 \text{ implies } \langle A v,\rho\rangle \geq 0 \quad \text{ for all } v\in \mathcal{D}(A)\cap V^+, \rho\in (V^*)^+.
\end{equation*}

\begin{theorem}[Arendt--Chernoff--Kato, \cite{Evans_Hanche-Olsen_1979, Arendt_Chernoff_Kato_1982, Koshkin_2013}] \label{theorem:Arendt_Chernoff_Kato}
  Let $V$ be a function system with unit $e$, let $\{\Phi(t)\}_{t\geq 0}$ be strongly  continuous semigroup with with generator $A$. Then the following conditions are equivalent:
  \begin{enumerate}
    \item $\Phi$ is a unital positive semigroup.
    \item $A$ is conditionally positive and $A[e] = 0$.
    \item The resolvent $R(\lambda,A)$ is positive for arbitrary large $\lambda > 0$ (equivalently, for all $\lambda > 0$), and $R(\lambda,A)[e] = \frac{1}{\lambda} e$.
  \end{enumerate}
\end{theorem}

In case $V$ is an operator system we  will require that $\Phi$ is a unital completely positive semigroup. The generator $A$ is then called \emph{conditionally completely positive} if all the amplifications $A^{(n)}$ are conditionally positive.

Unital completely positive semigroups are completely  contractive, for them it holds that $s(A)=\omega_0=0$, therefore the resolvent is defined for all $\lambda$ such that $\mathrm{Re}(\lambda)>0$ by formula \eqref{eq:resolvent_as_Laplace}.

\subsection{Injectivity} \label{subsec:injectivity}

An operator system $Z$ is called \emph{injective} if given a unital completely positive map $\phi : V \rightarrow Z$ and an embedding (complete isometry) $i : V \hookrightarrow W$ there is an completely positive map $\psi : W \rightarrow Z$ which is an extension of $\phi$, i.e. such that the diagram commutes
\begin{equation*}
  \begin{tikzcd}
  V \arrow[d, "i"', hook] \arrow[r, "\phi"] & Z \\
  W \arrow[ru, "\psi"']                     &
  \end{tikzcd}
\end{equation*}
It turns out the injective operator systems are also injective as operator  spaces \cite[Proposition 15.1]{Paulsen_2003}, i.e. every completely bounded map $\phi : V \rightarrow Z$ can be extended to
$\psi : W \rightarrow Z$ so that $\lVert\psi\rVert_{cb} = \lVert\phi\rVert_{cb}$.

An important theorem of Arveson states that the algebra $\mathcal{B}(\mathcal{H})$ is injective.
\begin{theorem}[Arveson, \cite{Arveson_1969}]
  Given Hilbert spaces $\mathcal{H}$ and $\mathcal{K}$, and concrete  operator systems $W \subseteq V$ on $\mathcal{H}$, any unital completely positive map $\phi : W \rightarrow \mathcal{B}(\mathcal{H})$ may be extended to a map $\psi : V \rightarrow \mathcal{B}(\mathcal{H})$.
\end{theorem}

The concrete operator system $V \subseteq \mathcal{B}(\mathcal{H})$ is injective if and only if there is a unital completely positive projection $p : \mathcal{B}(\mathcal{H}) \twoheadrightarrow \mathcal{B}(\mathcal{H})$ with $\mathrm{Im}(p) = V$, which is the extension of identity map $\mathrm{id}_V$. By the result of Choi and Effros \cite[Theorem 15.2]{Paulsen_2003}, each injective operator system is completely order isomorphic to a $C^*$-algebra (in fact a motonote complete $C^*$-algebra \cite{Saito_Wright_2015}) with multiplication induced by the projection $p$. Therefore it holds that the injective matricial systems are exactly the finite-dimensional $C^*$-algebras.

The \emph{extension} of an operator system $V$ is an operator system $W$ together with a unital completely positive embedding $i : V \hookrightarrow W$. The extension is \emph{injective} if $W$ is an injective operator system. The extension is \emph{essential} in case the map $\psi : W \rightarrow Z$ is an embedding if and only if $\psi \circ i : V \rightarrow Z$ is. The extension is called \emph{rigid} if for any $\psi : W \rightarrow W$ the condition $\psi \circ i = i$ implies $\psi = \mathrm{id}_W$. That means, the unique extension of $\mathrm{id}_V$ is $\mathrm{id}_W$.

The \emph{injective envelope} is a minimal injective  extension of $V$. That means, if $W$ is an injective envelope of $V$ with embedding $i : V\hookrightarrow W$, then the only injective operator system $Z$ such that $i[V]\subseteq Z\subseteq W$ is $Z = W$. The construction of injective envelope was proposed by Hamana \cite{Hamana_1979}. It turns out that injective envelope is a rigid and essential extension, and vice versa, if $W$ is an injective rigid extension of $V$, then it is an injective envelope.

We will be interested in extension properties of semigroups.  If $i : V \hookrightarrow W$ is an embedding, and $\Phi$ is a semigroup on $V$, then a semigroup $\Psi$ on $W$ is an extension if it holds that $\Psi(t) \circ i = i \circ \Phi(t)$ for all $t \geq 0$. If there also exists a quotient map $p : W \twoheadrightarrow V$, such that $p\circ i = \mathrm{id}_V$ then $\Psi$ is called \emph{co-extension} if $p\circ\Psi(t) = \Phi(t)\circ p$ and \emph{dilation} if $\Phi(t) = p\circ\Psi(t)\circ i$. Note that if $\Psi$ is a dilation, then it also is an extension and co-extension of $\Phi$. For convenience of the reader we present Figure~\ref{fig:extensions}, which illustrates differences between the concepts of extensions, co-extensions and dilations.
\begin{figure}[h]
  \centering
  \begin{tabular}[b]{c}
    \begin{tikzcd}
    V \arrow[r, "\Phi(t)"] \arrow[d, "i"', hook] & V \arrow[d, "i", hook] \\
    W \arrow[r, "\Psi(t)"']                      & W
    \end{tikzcd}
    \\
    extension
  \end{tabular}
  \qquad
  \begin{tabular}[b]{c}
    \begin{tikzcd}
    V \arrow[r, "\Phi(t)"]                            & V                            \\
    W \arrow[r, "\Psi(t)"'] \arrow[u, "p", two heads] & W \arrow[u, "p"', two heads]
    \end{tikzcd}
    \\
    co-extension
  \end{tabular}
  \qquad
  \begin{tabular}[b]{c}
    \begin{tikzcd}
    V \arrow[r, "\Phi(t)"] \arrow[d, "i"', hook] & V                            \\
    W \arrow[r, "\Psi(t)"']                      & W \arrow[u, "p"', two heads]
    \end{tikzcd}
    \\
    dilation
  \end{tabular}
  \caption{The commutative diagrams of extensions, co-extensions and dilations. Here $i : V \hookrightarrow W$ is a unital completely positive embedding and $p : W \twoheadrightarrow V$ is a unital completely positive quotient map. }
  \label{fig:extensions}
\end{figure}
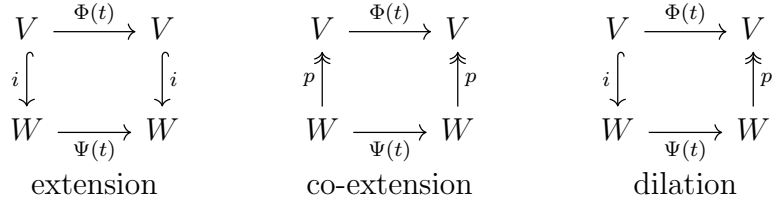

\section{Main result} \label{sec:main_result}

Given some semigroup $\Phi$ on operator system $V$, we want to somehow extend it a some semigroup $\Psi$ on $C^*$-algebra. In this section we prove that it is always possible to extend the continuous semigroups in such a way in case of matricial systems. Then, we also show that the extension is unique for invertible semigroups.

\subsection{Extension of semigroups}

Let us first note that this problem is trivial  for discrete-time semigroups.
\begin{proposition} \label{remark:discrete_case}
  Let $V$ be an operator system, $W$ it's injective  extension, and $i : V \hookrightarrow W$ an embedding. Let $\Phi : \mathbb{Z}_+ \rightarrow \mathrm{UCP}(V)$ be a discrete-time operator semigroup on $V$. Then it can   be extended to a semigroup $\Psi : \mathbb{Z}_+ \rightarrow \mathrm{UCP}(W)$.
\end{proposition}
\begin{proof}
  Discrete-time one-parameter semigroups are determined by their element $\Phi(1)$, because $\Phi(n) = \Phi(1)^n$. Using the injectivity of $W$, extend $\Phi(1) : V \rightarrow V$ to the operator $\Psi(1) : W \rightarrow W$.
\end{proof}

Let us state the result for continuous-time semigroups over matricial systems.
\begin{theorem} \label{theorem:extension_of_semigroups}
  Let $V$ be a matricial system and let $W$ be it's injective envelope with embedding $i : V \hookrightarrow W$. Let $\Phi : \mathbb{R}_+ \rightarrow \mathrm{UCP}(V)$ be a continuous unital completely positive one-parameter semigroup on $V$. Then it can be extended to a continuous unital completely positive semigroup $\Psi : \mathbb{R}_+ \rightarrow \mathrm{UCP}(W)$ on $W$.
\end{theorem}
Before proving the theorem, let us give some observations.
\begin{observation}
  Since we consider the matricial (i.e. finite-dimensional) case, the space $\mathrm{UCP}(V,W)$ of unital completely positive maps is compact in the norm topology $\lVert\cdot\rVert_{cb}$. Also, the space
  \begin{equation*}
    \left\{\psi\in\mathrm{UCP}(W) \, \middle| \, \psi\circ i = i\circ \phi\right\}
  \end{equation*}
  of the extensions of a given $\phi\in\mathrm{UCP}(V)$ is compact.
\end{observation}
\begin{observation}
  Let $G$ be a generator of some unital completely positive semigroup on some operator system $W$. Then for all $\lambda > 0$ it holds that $\lambda R(\lambda,G) \in \mathrm{UCP}(W)$ is unital and completely positive. If $\lambda,\mu > 0$ are two sufficiently close numbers, then by Hilbert's identity \eqref{eq:Hilberts_identity} it holds
  \begin{equation*}
    \lambda R(\lambda,G) = \sum_{k=0}^\infty \frac{\lambda}{\mu} \left(1-\frac{\lambda}{\mu}\right)^k (\mu R(\mu,G))^{k+1}.
  \end{equation*}
  Let $\beta$ be a number such that $0<\beta\leq1$. Let us examine the properties of a nonlinear map $H_\beta$
  \begin{equation*}
    H_\beta : \mathrm{UCP}(W) \rightarrow \mathrm{UCP}(W), \quad \phi \mapsto \sum_{k=0}^\infty \beta (1-\beta)^k \phi^{k+1}.
  \end{equation*}
  First of all note that the image of any unital map is unital,
  \begin{equation*}
    H_\beta[\phi][e] = \sum_{k=0}^\infty \beta (1-\beta)^k \phi^{k+1}[e] = e.
  \end{equation*}
  and the image of any positive map is positive as the sum of positive maps.

  Next, the map is continuous: let $\{\phi_\alpha\}_\alpha$ be a net of unital completely positive maps converging to $\phi\in\mathrm{UCP}(W)$, i.e. $\lim_\alpha \lVert\phi-\phi_\alpha\rVert_{cb} = 0$. Then
  \begin{equation*}
    \lVert \phi^{k+1} - \phi_\alpha^{k+1} \rVert_{cb} = \lVert \sum_{l=0}^{k} \phi^{k-l} (\phi-\phi_\alpha) \phi_\alpha^l \rVert_{cb} \leq (k+1)\lVert\phi-\phi_\alpha\rVert_{cb}
  \end{equation*}
  and
  \begin{equation*}
  \begin{aligned}
    \lVert H_\beta[\phi] - H_\beta[\phi_\alpha] \rVert_{cb}
    &\leq \sum_{k=0}^\infty \beta (1-\beta)^k \lVert\phi^{k+1}-\phi_\alpha^{k+1}\rVert_{cb} \\
    &\leq \sum_{k=0}^\infty \beta (1-\beta)^k (k+1) \lVert\phi-\phi_\alpha\rVert_{cb} \\
    &= \frac{1}{\beta} \lVert\phi-\phi_\alpha\rVert_{cb},
  \end{aligned}
  \end{equation*}
  so $\lim_\alpha H_\beta[\phi_\alpha] = H_\beta[\phi]$ uniformly on $\beta$ inside some interval $[\varepsilon,1]$, where $\varepsilon>0$. The map $H_\beta$ between compact Hausdorff spaces is continuous and injective,  therefore the image $\mathrm{Im}(H_\beta)$ is a compact subspace of $\mathrm{UCP}(W)$.

  Also, if $\beta_1, \beta_2 \in (0,1]$, then  $H_{\beta_1\cdot\beta_2} = H_{\beta_1}\circ H_{\beta_2}$.
\end{observation}

Now, we turn to the proof of Theorem~\ref{theorem:extension_of_semigroups}.
\begin{proof}
  Let $A$ be a generator of a semigroup $\Phi : \mathbb{R}_+ \rightarrow \mathrm{UCP}(V)$. Let us define for arbitrary $\omega > 0$ the set $\mathcal{G}_\omega$ of functions $(0,\omega] \rightarrow \mathrm{UCP}(W)$ which extend $\lambda \mapsto \lambda R(\lambda,A)$ and satisfy an analogue of Hilbert's identity \eqref{eq:Hilberts_identity}
  \begin{equation*}
    \mathcal{G}_\omega = \left\{ F : (0,\omega] \rightarrow \mathrm{UCP}(W) \, \middle| \;
    \begin{matrix}
    F(\lambda)\circ i = i \circ \lambda R(\lambda,A), \\
    \frac{\lambda F(\mu) - \mu F(\lambda)}{\lambda-\mu} = F(\lambda)F(\mu), \\
    \text{ for all } \lambda,\mu \in (0,\omega]. \\
    \end{matrix}
  \right\}
  \end{equation*}
  Such functions may be generated in the following way: let $F(\omega) \in \mathrm{UCP}(W)$ be some extension of $\omega R(\omega,A)$, which exists since $W$ is injective. Then the function $F$ is defined for all $\lambda\in (0,\omega]$ as
  \begin{equation*}
    F(\lambda) = H_{\lambda/\omega}[F(\omega)].
  \end{equation*}
  Therefore the set $\mathcal{G}_\omega$ is not empty. On the other hand, given $F\in\mathcal{G}_\omega$, the map $F(\omega)$ is unital
  completely positive and extends $\omega R(\omega,\lambda)$. Let us endow $\mathcal{G}_\omega$ with topology, induced by this bijection
  \begin{equation*}
    \mathcal{G}_\omega \cong \left\{ \psi\in\mathrm{UCP}(W) \, \middle| \, \psi\circ i = i\circ\omega R(\omega,A) \right\},
  \end{equation*}
  then $\mathcal{G}_\omega$ is a Hausdorff compact space. This topology may alternatively be described as a topology of uniform convergence on each closed interval $[\varepsilon,\omega]$ for $\varepsilon>0$.

  If $\omega_1 < \omega_2$, then the space of functions $\mathcal{G}_{\omega_2}$ is embedded into the space $\mathcal{G}_{\omega_1}$ via the map $H_{\omega_1/\omega_2}$.
  \begin{equation*}
    \text{If } \omega_1 < \omega_2, \text{ then } \mathcal{G}_{\omega_1} \supseteq \mathcal{G}_{\omega_2}.
  \end{equation*}
  So, the family $\{\mathcal{G}_\omega\}_{\omega>0}$ is a decreasing net of nonempty compact Hausdorff spaces, therefore it has a nonempty intersection. Let us denote
  \begin{equation*}
    \mathcal{G}_\infty = \bigcap_{\omega>0} \mathcal{G}_\omega.
  \end{equation*}

  Let $F\in \mathcal{G}_\infty$ be some function from $\mathcal{G}_\omega$. This function is defined for all $\lambda > 0$, so that $F(\lambda)\in\mathrm{UCP}(W)$, it satisfies an analogue of Hilbert's identity and is an extension of $\lambda \mapsto \lambda R(\lambda,A)$. Let us show that $F(\lambda) \rightarrow \mathrm{id}_W$ when $\lambda \rightarrow \infty$. Let $\{\lambda_\alpha\}_\alpha$ be an increasing net of positive numbers, such that the net converges to some $J = \lim_\alpha F(\lambda_\alpha)$. Such nets exist, since $\mathrm{UCP}(W)$ is compact. Then it holds
  \begin{equation*}
  \begin{aligned}
    J\circ i
    &= \left\{\lim_\alpha F(\lambda_\alpha)\right\}\circ i = \lim_\alpha \left\{F(\lambda_\alpha)\circ i\right\} \\
    &= \lim_\alpha \left\{i\circ \lambda_\alpha R(\lambda_\alpha,A)\right\} = i\circ\left\{\lim_\alpha \lambda_\alpha R(\lambda_\alpha,A)\right\} \\
    &= i\circ \mathrm{id}_V.
  \end{aligned}
  \end{equation*}
  Because $W$ is a rigid extension, this implies $J = \mathrm{id}_W$. That means, all converging subnets of $\{F(\lambda)\}_{\lambda>0}$ have $\mathrm{id}_W$ as a limit, so $\lim_{\lambda\rightarrow \infty} F(\lambda) = \mathrm{id}_W$.

  By the general theory of one-parameter semigroups (see for example \cite[Theorem 2.5.13]{Butzer_Berens_1967}) it follows that $F(\lambda) = \lambda R(\lambda,G)$ for some generator $G$. More explicitely, we can define $G$ as
  \begin{equation*}
    G = \lambda (\mathrm{id}_W - F(\lambda)^{-1}).
  \end{equation*}
  This definition does not depend on $\lambda$ due to Hilbert's identity \eqref{eq:Hilberts_identity}. The generator $G$ is an extension of $A$, and the semigroup $\Psi(t) = e^{t G}$ is a continuous unital completely positive semiroup on $W$ that extends $\Phi$.
\end{proof}

Let us note that the proof of Theorem~\ref{theorem:extension_of_semigroups} significantly relies on the finite-dimentionality of $W$. To be exact, it relies on the fact that $\mathrm{W}$ is compact in norm topology. Also note that all the extensions are described in terms of a compact $\mathcal{G}_\infty$.

\subsection{Extension of groups}

It turns out that the extension given by Theorem~\ref{theorem:extension_of_semigroups} is unique in case $\Phi$ corresponds to the invertible dynamics, i.e. it is a one-parameter group.

First of all, we will need a simple lemma about rigid extensions.
\begin{lemma} \label{lemma:rigidity_implies_inversion}
  Let $V$ be an operator system and $W$ it's rigid extension with embedding $i : V \hookrightarrow W$. If $\phi_1,\phi_2 : V \rightarrow V$ are inverse maps
  \begin{equation*}
    \phi_1\circ\phi_2 = \phi_2\circ\phi_1 = \mathrm{id}_V,
  \end{equation*}
  then any of their extensions $\psi_1,\psi_2 : W\rightarrow W$ are also inverse
  \begin{equation*}
    \psi_1\circ\psi_2 = \psi_2\circ\psi_1 = \mathrm{id}_W.
  \end{equation*}
  Consequentially, the extension of invertible map is unique, if it exists.
\end{lemma}
\begin{proof}
  The maps $\psi_1\circ\psi_2$ and $\psi_2\circ\psi_1$ extend the identity $\mathrm{id}_V$
  \begin{equation*}
    \psi_1\circ\psi_2\circ i = \psi_2\circ\psi_1\circ i = i\circ\mathrm{id}_V = i,
  \end{equation*}
  therefore by rigidity are equal to $\mathrm{id}_W$.

  If $(\psi_1, \psi_2)$ and $(\tilde{\psi}_1, \tilde{\psi}_2)$ are two pairs of inverses that extend $(\phi_1,\phi_2)$, then the pairs $(\psi_1,\tilde{\psi}_2)$ and $(\tilde{\psi}_1,\psi_2)$ are also inverses. The inverse of a map is always unique, therefore $\psi_1 = \tilde{\psi}_1$, $\psi_2 = \tilde{\psi}_2$, and the maps $\psi_1$, $\psi_2$ are unique as extensions of $\phi_1$, $\phi_2$.
\end{proof}

This lemma is connected with the result of Hamana \cite{Hamana_1985} (see also \cite[Corollary 8.1.23]{Saito_Wright_2015}) about the extensions of automorphism groups. Let us denote $\mathrm{Aut}(V)$ the group of unital completely positive automorphisms on operator system $V$. If the operator system $W$ is a $C^*$-algebra, then $\mathrm{Aut}(W)$ is a group of algebraic $*$-automorphisms on $W$ \cite[Corollary 3.2]{Choi_1974}.

\begin{theorem}[Hamana, \cite{Hamana_1985}] \label{theorem:automorphisms}
  Let $V$ be an operator system and $W$ it's injective extension with the embedding $i : V \hookrightarrow W$. Then there exists an injective homomorphism of groups $\chi : \mathrm{Aut}(V) \hookrightarrow \mathrm{Aut}(W)$. That means, any group of unital completely positive automorphisms on $V$ uniquely extends to a group of algebraic $*$-automorphisms on $W$. Also, if $\psi\in\mathrm{W}$ is an automorphism that leaves $V$ invariant, then $\psi = \chi(\psi\vert_V)$.
\end{theorem}
\begin{proof}
  Let us construct the map $\chi$. Let us fix some $\phi\in\mathrm{Aut}(V)$. Because the extension $W$ is injective, there is some $\psi\in\mathrm{UCP}(W)$ such that $\psi\circ i = i\circ\phi$. Because $W$ is a rigid extension, by the Lemma~\ref{lemma:rigidity_implies_inversion} the extension $\psi$ is unique. Then, set $\chi(\phi) = \psi$.
  \begin{equation*}
    \begin{tikzcd}
    V \arrow[r, "\phi"] \arrow[d, "i"', hook] & V \arrow[d, "i", hook] \\
    W \arrow[r, "\chi(\phi)"']                & W
    \end{tikzcd}
  \end{equation*}

  Let us show that $\chi$ is a homomorphism. If $\phi_1,\phi_2 \in\mathrm{Aut}(V)$, then both $\chi(\phi_1\circ\phi_2)$ and $\chi(\phi_1)\circ\chi(\phi_2)$ extend $\phi_1\circ\phi_2$, therefore they are equal. The map $\chi$ is injective, because the following chain of equations hold:
  \begin{equation*}
  \begin{gathered}
    \chi(\phi_1) = \chi(\phi_2), \\
    \chi(\phi_1)\circ i = \chi(\phi_2)\circ i, \\
    i\circ\phi_1 = i\circ\phi_2, \\
    \phi_1 = \phi_2.
  \end{gathered}
  \end{equation*}
  If $\psi\in\mathrm{Aut}(W)$ preserves the subspace $V$, then $\psi = \chi(\psi\vert_V)$ by uniqueness.
\end{proof}

\begin{theorem} \label{theorem:extension_of_groups}
  Let $V$ be a matricial system and let $W$ be it's injective envelope with embedding $i : V \hookrightarrow W$. Let $\Phi : \mathbb{R} \rightarrow \mathrm{Aut}(V)$ be a continuous group of unital completely positive automorphism on $V$. Then it can be extended to a continuous group of algebraic $*$-automorphisms $\Psi : \mathbb{R} \rightarrow \mathrm{Aut}(W)$ on $W$, in a unique way.
\end{theorem}

\begin{proof}
  By the Theorem~\ref{theorem:extension_of_semigroups} there exist semigroups $\Psi_+,\Psi_- : \mathbb{R}_+ \rightarrow \mathrm{UCP}(W)$ such that
  \begin{equation*}
    \Psi_+(t)\circ i = i\circ \Phi(t), \qquad \Psi_-(t)\circ i = i\circ \Phi(-t) \quad\text{ for all }t\geq 0.
  \end{equation*}
  By the Lemma~\ref{lemma:rigidity_implies_inversion}, these semigroups are inverse to each other, $\Phi_-(t)\circ\Psi_+(t) = \mathrm{id}_W$ for all $t\geq 0$, therefore they define a continuous one-parameter group $\Psi : \mathbb{R} \rightarrow \mathrm{Aut}(W)$ so that $\Psi(t) = \Psi_+(t)$ and $\Psi(-t) = \Psi_-(t)$ for all $t\geq 0$. By the Theorem~\ref{theorem:automorphisms} this extension is unique and
  \begin{equation*}
    \Psi = \chi\circ\Phi : \mathbb{R} \rightarrow \mathrm{Aut}(W).
  \end{equation*}
\end{proof}
Also note that the generator $G$ of a group $\Psi$ is a derivation on an algebra $W$ and the group preserves pure states.

\section{Discussion} \label{sec:discussion}

The theory of completely positive semigroups on $C^*$-algebras is a classical field of study in the context of quantum probability and quantum stochastic processes \cite{Kraus_1983, Breuer_Petruccione_2002}. The most important results include the statement that the generators of uniformly continuous completely positive semigroups on $C^*$-algebras have a GKSL structure \cite{Lindblad_1976, Gorini_Kossakowski_Sudarchan_1976}. At the same time, quite a little is known about the general structure of completely positive semigroups with unbounded generators \cite{Siemon_Holevo_Werner_2017}. Another important result is concerned with the idea that completely positive semigroups on $C^*$-algebra can always be dilated to a semigroup of $*$-homomorphisms \cite{Bhat_1996,Bhat_1999,Gaebler_2013}.

Operator systems can be though of as natural generalisation of a category of $C^*$-algebras with completely positive maps, and the study of state spaces of operator systems can be naturally considered as ``non-commutative'' convex geometry \cite{Davidson_Kennedy_2019, Kennedy_Shamovich_2021}. Operator systems are also studied in the context of Generalized Probabilistic Theories \cite{Plavala_2021,Aubrun_Lami_Palazuelos_Plavala_2021}.

The Theorems~\ref{theorem:extension_of_semigroups} and \ref{theorem:extension_of_groups} therefore help us to make sure that any evolution on matricial systems has some physical meaning as Markovian quantum stochastic processes. Let us illustrate it in the context of rebit calculations.

Let us examine the most simple case  of matricial systems inside of $2\times2$-matrices $\mathbb{M}_{2}$. Let $I,X,Y,Z$ be the Pauli matrices
\begin{equation*}
  I = \begin{bmatrix}1&0\\0&1\end{bmatrix}, \quad
  X = \begin{bmatrix}0&1\\1&0\end{bmatrix}, \quad
  Y = \begin{bmatrix}0&-i\\ i&0\end{bmatrix}, \quad
  Z = \begin{bmatrix}1&0\\0&-1\end{bmatrix}.
\end{equation*}

Let $V$ be an operator system and $W$ it's injective envelope. There are $4$ different cases of operator systems that are embedded into $\mathbb{M}_{2}$.
\begin{enumerate}
  \item $\dim{V}=1$, then $V = \mathrm{span}\{I\}$ and $W = \mathbb{C}$. The trivial case.
  \item $\dim{V}=2$, for example $V=\mathrm{span}\{I,Z\}$ -- the operator system of diagonal $2\times 2$-matrices. It's envelope is a commutative algebra $W=\mathbb{C}^2$ and this system cannot posses any ``quantumness''.
  \item $\dim{V}=3$, for example $V=\mathrm{span}\{I,X,Z\}$ -- the operator system spanned by symmetric $2\times 2$-matrices with real coefficients. This system is often called a \emph{rebit} (real bit). It's envelope is $W=\mathbb{M}_{2}$.
  \item $\dim{V}=4$, then $V=\mathrm{span}\{I,X,Y,Z\}=\mathbb{M}_{2}$ is the system of one qubit, $W = \mathbb{M}_{2}$.
\end{enumerate}

Out of the four cases the most interesting is a case of rebit $V=\mathrm{span}\{I,X,Z\}$. The set of positive elements is
\begin{equation*}
  V^+ = \left\{ a I + b X + c Z = \begin{bmatrix} a+c & b \\ b & a-c \end{bmatrix} \; \middle| \, b^2 + c^2 \leq a^2 \right\}.
\end{equation*}
The Theorem~\ref{theorem:extension_of_semigroups} implies that any evolution on a rebit is determined by some evolution on a qubit which preserves the subspace of real-valued matrices.

Geometrically speaking, the state space $S(V)$ of a rebit is a disk, and the pure states are a boundary circle of that disk. The invertible dynamics on that disk uniquely correspond to the rotations of that disk, and by the Theorem~\ref{theorem:extension_of_groups} they and are described by the qubit Hamiltonian of form $\frac{\omega}{2} Y$, where $\omega\in\mathbb{R}$ is a rotation speed. That is, any automorphism group $\Phi$ on a rebit $V$ with a generator $A$ is uniquely extended to a group $\Psi$ on a qubit which has a generator $G = i \frac{\omega}{2} [ Y, \cdot\,]$, so that
\begin{equation*}
  A[a I + b X + c Z] = G[a I + b X + c Z] = \omega (- c X + b Z),
\end{equation*}
and
\begin{multline*}
  \Phi(t)[a I + b X + c Z] = \Psi(t)[a I + b X + c Z] = \\
    = a I + \{b \cos(\omega t)- c \sin(\omega t)\} X + \{b\sin(\omega t) + c\cos(\omega t)\} Z.
\end{multline*}

Now, let us consider the purely dissipative dynamics on a rebit. Let $\Delta > 0$ be a constant of dissipation and let the dynemics be of the form
\begin{equation*}
\begin{gathered}
  A[a I + b X + c Z] = - \Delta b X - \Delta c Z,\\
  \Phi(t)[a I + b X + c Z] = a I + b e^{-\Delta t} X + c e^{-\Delta t} Z.
\end{gathered}
\end{equation*}
By the Theorem~\ref{theorem:extension_of_semigroups}, this semigroup can be extended to a semigroup on a qubit. However, the extension is not unique. On the one hand, this evolution is described by a GKSL generator
\begin{equation*}
  G_1[B] = \Delta\left(\frac{1}{2} X B X + \frac{1}{2} Z B Z - B\right),
\end{equation*}
on the other hand, it is also possible to use the generator
\begin{equation*}
  G_2[B] = \frac{4}{3}\Delta\left(\frac{1}{3} X B X + \frac{1}{3} Y B Y + \frac{1}{3} Z B Z - B\right).
\end{equation*}
Both of these generators act on a rebit in the same way, but they are not equivalent on a qubit:
\begin{equation*}
\begin{gathered}
  A[a I + b X + c Z] = G_1[a I + b X + c Z] = G_2[a I + b X + c Z], \\
  G_1[Y] = -2 \Delta Y, \quad G_2[Y] = - \Delta Y.
\end{gathered}
\end{equation*}

More generally, given a $d$-dimensional Hilbert space, one can consider the matricial system spanned by real matrices
\begin{equation*}
  V = \mathrm{span}\left\{ v\in\mathbb{M}_{d} \, \middle| \, v \text{ is a symmetric matrix with real entries } \right\}.
\end{equation*}
That operator system could be understood as representing the $d$-dimensional ``real-valued'' quantum mechanics. The Theorems~\ref{theorem:extension_of_semigroups} and \ref{theorem:extension_of_groups} then imply that ``real-valued'' physical system $V$ corresponds to the ``complex-valued'' system $W = \mathbb{M}_{d}$. On the other hand, it is known that any quantum computation on $n$ qubits may be represented as a computation on $n+1$ rebits \cite{Rudolph_Grover_2002,Delfrosse_et_al_2015}.

The operator systems can also be treated as ``non-commutative'' versions of graphs \cite{Duan_Severini_Winter_2012}, and this point of view has applications to the theory of quantum error correction \cite{Amosov_Mokeev_Pechen_2019,Amosov_Mokeev_Pechen_2021,Amosov_Mokeev_2020}. The presented theorems may be of interest to the study of dynamics on such graphs.

Another interesting direction concerns the application of Theorem~\ref{theorem:automorphisms} together with Theorem~\ref{theorem:extension_of_groups}. These theorems state that any group of continuous symmetries on operator system is induced by an authomorphism group of overlying $C^*$-algebra. Therefore the embeddability properties of operator systems \cite{Fritz_Netzer_Thom_2017} could be characterized by it's symmetry group of completely positive maps.

\section*{Acknowledgements}
  I am grateful to G. G. Amosov for useful discussion and comments.
The work is supported by Russian Science Foundation under the grant
no. 19-11-00086 and performed in Steklov Mathematical Institute of
Russian Academy of Sciences.

\bibliographystyle{ieeetr}
\bibliography{bibliography}

\begin{thebibliography}{10}

\bibitem{Breuer_Petruccione_2002}
H.-P. Breuer and F.~Petruccione, {\em The theory of open quantum systems}.
\newblock Oxford University Press, 2002.

\bibitem{Engel_Nagel_2000}
K.-J. Engel, R.~Nagel, and S.~Brendle, {\em One-parameter semigroups for linear
  evolution equations}.
\newblock Graduate Texts in Mathematics 194, Springer-Verlag New York, 1~ed.,
  2000.

\bibitem{Batkai_Fijvacz_Rhandi_2017}
A.~B{\'a}tkai, M.~Fijav{\v{z}}, and A.~Rhandi, {\em Positive operator
  semigroups: from finite to infinite dimensions}.
\newblock Operator Theory: Advances and Applications 257, Birkhäuser Basel,
  1~ed., 2017.

\bibitem{Kraus_1983}
K.~Kraus, A.~B{\"o}hm, J.~Dollard, and W.~Wootters, {\em States, effects, and
  operations: fundamental notions of quantum theory}.
\newblock Lecture notes in physics, Springer-Verlag, 1983.

\bibitem{Choi_Effros_1977}
M.-D. Choi and E.~Effros, ``Injectivity and operator spaces,'' {\em Journal of
  Functional Analysis}, vol.~24, no.~2, pp.~156--209, 1977.

\bibitem{Paulsen_2003}
V.~Paulsen, {\em Completely bounded maps and operator algebras}.
\newblock Cambridge University Press, 2003.

\bibitem{Helemskii_2010}
A.~Helemskii, {\em Quantum functional analysis: non-coordinate approach}.
\newblock University Lecture Series 56, American Mathematical Society, 2010.

\bibitem{Lindblad_1976}
G.~Lindblad, ``On the generators of quantum dynamical semigroups,'' {\em
  Communications in Mathematical Physics}, vol.~48, no.~2, pp.~119--130, 1976.

\bibitem{Gorini_Kossakowski_Sudarchan_1976}
V.~Gorini, A.~Kossakowski, and E.~Sudarshan, ``Completely positive dynamical
  semigroups of {N}-level systems,'' {\em Journal of Mathematical Physics},
  vol.~17, no.~5, pp.~821--825, 1976.

\bibitem{Bhat_1996}
B.~Bhat, ``An index theory for quantum dynamical semigroups,'' {\em
  Transactions of the American Mathematical Society}, vol.~348, no.~2,
  pp.~561--583, 1996.

\bibitem{Bhat_1999}
B.~Bhat, ``Minimal dilations of quantum dynamical semigroups to semigroups of
  endomorphisms of {$C^*$}-algebras,'' {\em Journal of the Ramanujan
  Mathematical Society}, vol.~14, no.~2, pp.~109--124, 1999.

\bibitem{Aliprantis_Chalambos_Tourky_2007}
C.~Aliprantis and R.~Tourky, {\em Cones and duality}, vol.~84.
\newblock American Mathematical Soc., 2007.

\bibitem{Effros_Ruan_2000}
E.~Effros and Z.-J. Ruan, {\em Operator spaces}.
\newblock London Mathematical Society Monographs New Series 23, Oxford
  University Press, USA, 2000.

\bibitem{Pisier_2003}
G.~Pisier, {\em Introduction to operator space theory}, vol.~294 of {\em London
  Mathematical Society Lecture Note Series}.
\newblock Cambridge University Press, text is free of markings~ed., 2003.

\bibitem{Paulsen_Todorov_Tomforde_2010}
V.~Paulsen, I.~Todorov, and M.~Tomforde, ``Operator system structures on
  ordered spaces,'' {\em Proceedings of the London Mathematical Society},
  vol.~102, p.~25–49, May 2010.

\bibitem{Paulsen_Tomforde_2009}
V.~Paulsen and M.~Tomforde, ``Vector spaces with an order unit,'' {\em Indiana
  University Mathematics Journal}, vol.~58, no.~3, pp.~1319--1359, 2009.

\bibitem{Kadison_1951}
R.~Kadison, {\em A representation theory for commutative topological algebra}.
\newblock American Mathematical Society, 1951.

\bibitem{Kakutani_1941}
S.~Kakutani, ``Concrete representation of abstract {$(M)$}-spaces ({A}
  characterization of the space of continuous functions),'' {\em Annals of
  Mathematics}, vol.~42, no.~4, pp.~994--1024, 1941.

\bibitem{Lindenstrauss_Tzafriri_1996}
J.~Lindenstrauss and L.~Tzafriri, {\em Classical {B}anach spaces {I}, {II}}.
\newblock Classics in Mathematics, Springer, reprint~ed., 1996.

\bibitem{Ruan_1988}
Z.-J. Ruan, ``Subspaces of {$C^*$}-algebras,'' {\em Journal of Functional
  Analysis}, vol.~76, no.~1, pp.~217--230, 1988.

\bibitem{Smith_1983}
R.~Smith, ``Completely bounded maps between {$C^*$}-algebras,'' {\em Journal of
  the London Mathematical Society}, vol.~s2-27, no.~1, pp.~157--166, 1983.

\bibitem{Kavruk_Paulsen_Todorov_Tomforde_2013}
A.~Kavruk, V.~Paulsen, I.~Todorov, and M.~Tomforde, ``Quotients, exactness, and
  nuclearity in the operator system category,'' {\em Advances in mathematics},
  vol.~235, pp.~321--360, 2013.

\bibitem{Arveson_2010}
W.~Arveson, ``The noncommutative {C}hoquet boundary {III}: operator systems in
  matrix algebras,'' {\em Mathematica Scandinavica}, pp.~196--210, 2010.

\bibitem{Evans_Hanche-Olsen_1979}
D.~E. Evans and H.~Hanche-Olsen, ``The generators of positive semigroups,''
  {\em Journal of Functional Analysis}, vol.~32, no.~2, pp.~207--212, 1979.

\bibitem{Arendt_Chernoff_Kato_1982}
W.~Arendt, P.~Chernoff, and T.~Kato, ``A generalization of dissipativity and
  positive semigroups,'' {\em Journal of Operator Theory}, pp.~167--180, 1982.

\bibitem{Koshkin_2013}
S.~Koshkin, ``A short proof of the {A}rendt-{C}hernoff-{K}ato theorem,'' {\em
  Archiv der Mathematik}, vol.~101, p.~143–147, Jul 2013.

\bibitem{Arveson_1969}
W.~Arveson, ``Subalgebras of {$C^*$}-algebras,'' {\em Acta Mathematica},
  vol.~123, no.~1, pp.~141--224, 1969.

\bibitem{Saito_Wright_2015}
K.~Sait{\^o} and J.~Wright, {\em Monotone complete {$C^*$}-algebras and generic
  dynamics}.
\newblock Springer, 2015.

\bibitem{Hamana_1979}
M.~Hamana, ``Injective envelopes of operator systems,'' {\em Publications of
  the Research Institute for Mathematical Sciences}, vol.~15, no.~3,
  pp.~773--785, 1979.

\bibitem{Butzer_Berens_1967}
P.~Butzer and H.~Berens, {\em Semi-groups of operators and approximation}.
\newblock Springer, 1967.

\bibitem{Hamana_1985}
M.~Hamana, ``Injecitve envelopes of {$C^*$}-dynamical systems,'' {\em Tohoku
  Mathematical Journal, Second Series}, vol.~37, no.~4, pp.~463--487, 1985.

\bibitem{Choi_1974}
M.-D. Choi, ``A schwarz inequality for positive linear maps on
  {$C^*$}-algebras,'' {\em Illinois Journal of Mathematics}, vol.~18, no.~4,
  pp.~565--574, 1974.

\bibitem{Siemon_Holevo_Werner_2017}
I.~Siemon, A.~Holevo, and R.~Werner, ``Unbounded generators of dynamical
  semigroups,'' {\em Open Systems \& Information Dynamics}, vol.~24, no.~04,
  p.~1740015, 2017.

\bibitem{Gaebler_2013}
D.~Gaebler, ``Unital dilations of completely positive semigroups: from
  combinatorics to continuity,'' 2013.

\bibitem{Davidson_Kennedy_2019}
K.~Davidson and M.~Kennedy, ``Noncommutative {C}hoquet theory,'' 2019.

\bibitem{Kennedy_Shamovich_2021}
M.~Kennedy and E.~Shamovich, ``Noncommutative {C}hoquet simplices,'' {\em
  Mathematische Annalen}, vol.~382, p.~1591–1629, Sep 2021.

\bibitem{Plavala_2021}
M.~Plávala, ``General probabilistic theories: an introduction,'' 2021.

\bibitem{Aubrun_Lami_Palazuelos_Plavala_2021}
G.~Aubrun, L.~Lami, C.~Palazuelos, and M.~Pl{\'a}vala, ``Entangleability of
  cones,'' {\em Geometric and Functional Analysis}, vol.~31, no.~2,
  pp.~181--205, 2021.

\bibitem{Rudolph_Grover_2002}
T.~Rudolph and L.~Grover, ``A 2 rebit gate universal for quantum computing,''
  2002.

\bibitem{Delfrosse_et_al_2015}
N.~Delfosse, P.~Allard~Guerin, J.~Bian, and R.~Raussendorf, ``Wigner function
  negativity and contextuality in quantum computation on rebits,'' {\em
  Physical Review X}, vol.~5, Apr 2015.

\bibitem{Duan_Severini_Winter_2012}
R.~Duan, S.~Severini, and A.~Winter, ``Zero-error communication via quantum
  channels, noncommutative graphs, and a quantum {L}ov{\'a}sz number,'' {\em
  IEEE Transactions on Information Theory}, vol.~59, no.~2, pp.~1164--1174,
  2012.

\bibitem{Amosov_Mokeev_Pechen_2019}
G.~Amosov, A.~Mokeev, and A.~Pechen, ``Non-commutative graphs and quantum error
  correction for a two-mode quantum oscillator,'' {\em Quantum Information
  Processing}, vol.~19, Feb 2020.

\bibitem{Amosov_Mokeev_Pechen_2021}
G.~Amosov, A.~Mokeev, and A.~Pechen, ``Noncommutative graphs based on
  finite-infinite system couplings: Quantum error correction for a qubit
  coupled to a coherent field,'' {\em Physical Review A}, vol.~103, no.~4,
  p.~042407, 2021.

\bibitem{Amosov_Mokeev_2020}
G.~Amosov and A.~Mokeev, ``On errors denerated by unitary dynamics of bipartite
  quantum systems,'' {\em Lobachevskii Journal of Mathematics}, vol.~41,
  p.~2310–2315, Dec 2020.

\bibitem{Fritz_Netzer_Thom_2017}
T.~Fritz, T.~Netzer, and A.~Thom, ``Spectrahedral containment and operator
  systems with finite-dimensional realization,'' {\em SIAM Journal on Applied
  Algebra and Geometry}, vol.~1, no.~1, pp.~556--574, 2017.

\end{thebibliography}

\end{document}